\documentclass[10pt]{amsart}

\usepackage[colorlinks]{hyperref}
\usepackage{color,graphicx,shortvrb}
\usepackage[latin 1]{inputenc}

\usepackage[active]{srcltx} 

\usepackage{enumerate}
\usepackage{amssymb}

\newtheorem{theorem}{Theorem}[section]

\newtheorem{corollary}[theorem]{Corollary}

\newtheorem{lemma}[theorem]{Lemma}

\newtheorem{problem}[theorem]{Problem}
\newtheorem{proposition}[theorem]{Proposition}

\def\J#1#2#3{ \left\{ #1,#2,#3 \right\} }

\def\NN{{\mathbb{N}}}
\def\11{\textbf{$1$}}
\def\CC{{\mathbb{C}}}

\DeclareGraphicsExtensions{.jpg,.pdf,.png,.eps}

\begin{document}

\title[A solution to Tingley's problem for compact C$^*$-algebras]{A solution to Tingley's problem for isometries between the unit spheres
of compact C$^*$-algebras and JB$^*$-triples}

\author[A.M. Peralta]{Antonio M. Peralta}

\address{Departamento de An{\'a}lisis Matem{\'a}tico, Facultad de
Ciencias, Universidad de Granada, 18071 Granada, Spain.}
\email{aperalta@ugr.es}

\author[R. Tanaka]{Ryotaro Tanaka}

\address{Faculty of Mathematics, Kyushu University, Fukuoka, 819-0395, Japan.}
\email{r-tanaka@math.kyushu-u.ac.jp}


\subjclass[2010]{Primary 47B49, Secondary 46A22, 46B20, 46B04, 46A16, 46E40, .}

\keywords{Tingley's problem; extension of isometries; JB$^*$-triples; compact operators}

\date{}

\begin{abstract} Let $f: S(E) \to S(B)$ be a surjective isometry between the unit spheres of two weakly compact JB$^*$-triples not containing direct summands of rank smaller than or equal to 3. Suppose $E$ has rank greater than or equal to 5. Applying techniques developed in JB$^*$-triple theory, we prove that $f$ admits an extension to a surjective real linear isometry $T: E\to B$. Among the consequences, we show that every surjective isometry between the unit spheres of two compact C$^*$-algebras $A$ and $B$ (and in particular when $A=K(H)$ and $B=K(H')$) extends to a surjective real linear isometry from $A$ into $B$. These results provide new examples of infinite dimensional Banach spaces where Tingley's problem admits a positive answer.
\end{abstract}

\maketitle
\thispagestyle{empty}

\section{Introduction}

Let $X$ and $Y$ be normed spaces, whose unit spheres are denoted by $S(X)$ and $S(Y)$, respectively. Suppose $f:S(X)\to S(Y)$ is a surjective isometry.  The so-called \emph{Tingley's problem} asks wether $f$ can be extended to a real-linear (bijective) isometry $T : X \to Y$ between the corresponding spaces (see \cite{Ting1987}). D. Tingley proved in \cite[THEOREM, page 377]{Ting1987} that every surjective isometry $f: X\to Y$ between the unit spheres of two finite dimensional spaces satisfies $f(-x) = -f(x)$ for every $x\in S(X)$.\smallskip

Mankiewicz established in \cite{Mank1972} that, given two convex bodies $V\subset X$ and $W\subset Y$, every surjective isometry $g$ from
$V$ onto $W$ can be uniquely extended to an affine isometry from $X$ onto $Y$. Consequently, every surjective isometry between the closed unit balls of two Banach spaces $X$ and $Y$ extends uniquely to a real-linear isometric isomorphism from $X$ into $Y$.\smallskip

Many authors have contributed with partial solutions to Tingley's problem for surjective isometries between the unit spheres of concrete Banach spaces. For example, we find affirmative answers to Tingley's problem in the setting of several classical real Banach spaces such as $\ell_p (\Gamma)$ (\cite{Di:p,Di:8,Di:1}), $L^p$-spaces (\cite{Ta:8,Ta:1,Ta:p}), and $C(X)$ spaces (\cite{Di:C,THL}). We refer to the monograph~\cite{Ding2009} for a good survey on the history of Tingley's problem. The solutions we know for this problem are based on interesting geometric ideas (compare \cite{Di,TL}). Tingley's problem was solved affirmatively in the case of finite dimensional polyhedral Banach spaces~\cite{KadMar2012}, and for the spaces belonging to a special class of Banach spaces called \emph{generalized lush spaces}~\cite{THL}. However, the problem is still open even in the simplest case of $X=Y$ and $\dim X=2$.\smallskip

Recently, the study of Tingley's problem on operator algebras was started by the second author of this note in \cite{Tan2016}, and we now have an affirmative answers in the case of surjective linear isometries between the unit spheres of finite dimensional C$^*$-algebras (see \cite{Tan2016-2}) and finite von Neumann algebras (cf. \cite{Tan2016preprint}). A key ingredient for the methods described in~\cite{Tan2016,Tan2016-2,Tan2016preprint} are results describing the facial structure of the unit ball of a $C^*$-algebra. Using the structure theorem for faces given by C.M. Edwards and G.T. R\"{u}ttimann~\cite{EdRutt88}, and C.A. Akemann and G.K. Pedersen~\cite{AkPed92}), we are in position to apply the characterization of the surjective isometries between the unitary groups of two von Neumann algebras given by O. Hatori and L. Moln\'{a}r~\cite{HM} to give an affirmative answer to Tingley's problem in the setting of finite von Neumann algebras.\smallskip

The purpose of this paper is to present some new partial solutions to Tingley's problem in the case of surjective isometries between the unit spheres of two weakly compact JB$^*$-triples not containing direct summands of rank smaller than or equal to 3. The class of JB$^*$-triples can be viewed as a Jordan generalization of the category of C$^*$-algebras. In the wider setting of JB$^*$-triples, a complete description of the norm closed faces of the closed unit ball was obtained by C.M. Edwards, C.S. Hoskin, F.J. Fern{\'a}ndez-Polo and the first author of this note in \cite{EdFerHosPe2010}. Using the result describing the facial structure of the closed unit ball of a JB$^*$-triple, we present here a completely new method to approach Tingley's problem in the Jordan setting.\smallskip

The main result in this note (see Theorem \ref{thm Tingley thm ofr weakly compact JB*-triples rank 5}) proves that every surjective isometry, $f: S(E) \to S(B)$, between the unit spheres of two weakly compact JB$^*$-triples not containing direct summands of rank smaller or equal than 3, where $E$ has rank greater or equal than 5, extends to a surjective real linear isometry $T: E\to B$. As an application, we particularly find an affirmative answer to Tingley's problem for surjective isometries between the unit spheres of two compact C$^*$-algebras (cf. Theorem \ref{thm Tingley compact Cstaralgebras}). We also prove that every surjective isometry between the unit spheres of two finite dimensional JB$^*$-triples, where the domain JB$^*$-triple has rank greater or equal than 5, extends to a surjective real linear isometry between the corresponding spaces (see Theorem \ref{thm Tingley thm ofr weakly compact JB*-triples rank 5 finite dimensional}). These results provide new examples of infinite dimensional Banach spaces where Tingley's problem admits a positive answer. The arguments in this note are completely new and independent from those studied in previous references. The results studying the geometry of JB$^*$-triples, and the Jordan theory itself provide a new point of view to tackle this open problem. Finally, we close the paper with an open problem which cannot be covered by our results.

\section{The contribution of Jordan theory to Tingley's problem}

We recall that a proper convex subset of the unit sphere of a normed space $X$ is said to be maximal if $C$ is not contained in any other proper maximal subset of $S(X)$. Let $f:S(X)\to S(Y)$ be a surjective isometry between the unit spheres of two Banach spaces. Applying Lemma 6.3 in \cite{Tan2016} we know that a convex set $C\subset S(X)$ is a maximal convex subset of $S(X)$ if and only if $T(C)$ satisfies the same property as a subset of $S(Y)$.\smallskip

Henceforth, the closed unit ball of a Banach space $X$ will be denoted by $X_1$. We also recall that a convex subset $F$ of $X_1$ is called a \emph{face} of $X_1$ if it satisfies the following property: a convex combination $t x + (1-t) y$, with $x,y \in X_1$ and $t\in [0,1]$, lies in $F$ if and only if $x$ and $y$ belong to $F$. It is worth to notice that a face $F$ of $X_1$ contains $0$ if and only if $F=X_1$, it is further known that every proper face of $X_1$ is contained in $S(X)$. Lemma 3.3 in \cite{Tan2016} shows how a suitable application of Eidelheit's separation theorem \cite[Theorem 2.2.26]{Megg98} proves that for every maximal convex subset $C$ of $S(X)$ there exists $\varphi$ in $S(X^*)$ satisfying  $C=\varphi^{-1} (\{1\}) \cap X_1$, in particular every maximal convex subset $C$ of $S(X)$ is a maximal proper norm closed face of $X_1$. Actually, by a result due to the second author of this note we know that a convex subset $C$ of the sphere of a Banach space $X$ is a maximal convex subset of $S(X)$ if and only if it is maximal as a proper face of $X_1$ (compare \cite[Lemma 3.2]{Tan2016preprint}).\smallskip

The arguments in the above two paragraphs can be applied to deduce the following lemma, which has been borrowed from \cite[Lemma 3.5]{Tan2014}.

\begin{lemma}\label{l 1.1}\cite[Lemma 3.5]{Tan2014} Let $f: S(X) \to S(Y)$ be a surjective isometry between the unit spheres of two Banach spaces, and let $C$ be a convex subset of $S(X)$. Then $C$ is a maximal proper face of $X_1$ if and only if $f(C)$ is a maximal proper closed face of $Y_1$.$\hfill\Box$
\end{lemma}

The results commented above reveal that, in order to attack Tingley's problem, the facial structure of the corresponding closed unit balls of the spaces plays a fundamental role. The main result in \cite{EdFerHosPe2010} culminate the complete description of the norm closed faces of the closed unit ball of a JB$^*$-triple. Before stating the characterization theorem, we shall recall some definitions needed for our purposes.\smallskip

\subsection{Facial structure of a JB$^*$-triple}

Let $X$ be a complex Banach space with dual space $X^*$. Let $X_1$
and $X^*_1$ denote the unit balls in $X$ and $X^*$, respectively. Given subsets $F$ of $X_1$ and $G$ of $X^*_1$, we set
\begin{equation}\label{1.1}
F^{\prime} = \{a \in X^*_1 : a(x) = 1\,\, \forall x \in F\},\quad G_{\prime} = \{x \in X_1 : a(x) = 1\,\, \forall a \in G\}.
\end{equation}
Then, $F^{\prime}$ is a weak$^*$-closed face of $X_1^*$ and
$G_{\prime}$ is a norm closed face of $X_1$. The subset $F$ of $X_1$ is said to be
a norm-semi-exposed face of $X_1$ if $F$ coincides with $(F^{\prime})_{\prime}$
and the subset $G$ of $X^*_1$ is said to be a weak$^*$-semi-exposed face of $X^*_1$ if
$G$ coincides with $(G_{\prime})^{\prime}$.\smallskip

JB$^*$-triples are Jordan Banach structures which generalize the abstract properties of C$^*$-algebras. The notion of JB$^*$-triple was introduced by W. Kaup in 1983, and provides a precise set of algebraic-analytic axioms to characterize when the open unit ball of a complex Banach space is a bounded symmetric domain (a property of holomorphic nature, see \cite{Ka83} for more details). The concrete definition tells that a JB$^*$-triple is a complex Banach space $E$ equipped with a
continuous triple product $\J ... : E\times E\times E \to E,$ which is conjugate linear in the middle variable and symmetric and bilinear in the outer variables satisfying the following axioms:
\begin{enumerate}[{\rm $(a)$}] \item (Jordan Identity) $L(a,b) L(x,y) - L(x,y) L(a,b)= L(L(a,b)x,y) - L(x, L(b,a) y),$ for all $a,b,x,y,$ in $E$, where $L(x,y)$ is the linear operator defined by $L(a,b) (z) =\{a,b,z\}$ ($\forall z\in E$);
\item The operator $L(a,a)$ is hermitian and has non-negative spectrum;
\item $\|\J aaa\| = \|a\|^3$, for every $a\in E$.\end{enumerate}\smallskip

L.A. Harris established in \cite{Harris74} that the open unit ball of a C$^*$-algebra $A$ is a bounded symmetric domain, and hence every C$^*$-algebra $A$ lies in the category of JB$^*$-triples. It is known that the triple product \begin{equation}\label{eq triple product on Cstaralg} (a,b,c)\mapsto \{a,b,c\} =1/2 (a b^* c + c b^* a), \ \ (a,b,c\in A),
 \end{equation} satisfies all the axioms in the definition of JB$^*$-triple. The category of JB$^*$-triples contains many other examples of Banach spaces which are not C$^*$-algebras. For example, given two complex Hilbert spaces $H_1$ and $H_2$, the triple product defined in \eqref{eq triple product on Cstaralg} equips the space $L(H_1,H_2)$, of all bounded linear operators between $H_1$ and $H_2$, with a structure of JB$^*$-triple. In particular, every complex Hilbert space if a JB$^*$-triple. JB$^*$-triples of the form $L(H_1,H_2)$ are called Cartan factors of type 1. Clearly the space $K(H_1,H_2)$, of all compact operators from $H_1$ into $H_2$ is a JB$^*$-subtriple of $L(H_1,H_2)$. Additional examples can be given by the rest of Cartan factors, which are defined as follows. Let $j$ be a conjugation (i.e. a conjugate linear isometry or period 2) on a complex Hilbert
space $H$, The assignment $x\mapsto x^t:=jx^*j$ defines a linear involution on $L(H).$ A Cartan factor of type 2 (respectively, of type 3) is a complex Banach space which coincide with the JB$^*$-subtriple of $L(H)$ of all $t$-skew-symmetric (respectively, $t$-symmetric) operators.\smallskip

A Cartan factor of type 4, is a a complex Hilbert space provided
with a conjugation $ x\mapsto \overline{x},$ where triple product
and the norm are  given by \begin{equation}\label{eq spin product}
\{x, y, z\} = (x|y)z + (z|y) x -(x|\overline{z})\overline{y},
\end{equation} and $ \|x\|^2 = (x|x) + \sqrt{(x|x)^2 -|
(x|\overline{x}) |^2},$ respectively. All we need to know about Cartan factors of types 5 and 6 is that they are finite dimensional (see \cite[Example 2.5.31]{Chu2012} for additional details).\smallskip

One of the most interesting properties of JB$^*$-triples was established by W. Kaup in \cite{Ka83} and assures that a (complex) linear bijection between JB$^*$-triples is a triple isomorphism if and only if it is an isometry (compare \cite[Proposition 5.5]{Ka83} or \cite[Theorem 2.2]{FerMarPe2004}). The same conclusion is no longer true for real linear isometries between JB$^*$-triples, for such a mapping $T$ we only know that $T$ preserves cubes, that is, $T(\{x,x,x\}) = \{T(x),T(x),T(x)\}$ (compare \cite{Da, ChuDaRuVen, IsKaRo95} and \cite{FerMarPe}). The conclusion in Tingley's theorem makes more useful the study of real linear isometric surjections between JB$^*$-triples.\smallskip

A JB$^*$-triple which is also a dual Banach space is called a JBW$^*$-triple. For example, every von Neumann algebra is a JBW$^*$-triple. All Cartan factors are JBW$^*$-triples. Every JBW$^*$-triple admits a unique isometric predual, and its triple product is separately weak$^*$
continuous \cite{BarTi86}. Furthermore, the second dual of a JB$^*$-triple $E$ is a JBW$^*$-triple under a triple product extending the product of $E$
\cite{Di86}.\smallskip

An element $e$ in a JBW$^*$-triple $E$ is called \emph{tripotent} if $\{e,e,e\}=e$. The set of all tripotents in $E$ will be denoted by $\mathcal{U}(E)$. In this case the eigenvalues of the operator $L(e,e)$ are precisely $0,1/2$ and $1$ and $E$ decomposes as the direct sum of the corresponding eigenspaces, that is, $$E= E_{2} (e) \oplus E_{1} (e) \oplus E_0 (e),$$ where for $i=0,1,2,$ $E_i (e)$ is the $\frac{i}{2}$ eigenspace of $L(e,e)$
(compare \cite[Definition 1.2.37]{Chu2012}). This decomposition is called the \emph{Peirce decomposition} of $E$ with respect to the
tripotent $e$, and the projection of $E$ onto $E_i(e)$, which is denoted by $P_i(e)$, is called the Peirce $i$ projection.\smallskip

The so-called \emph{Peirce arithmetic} asserts that, for $k,j,l\in \{0,1,2\}$ we have $$ \{ E_k(e) , E_j (e), E_l (e)\} \subseteq E_{k-j+l} (e),$$ if $k-j+l \in\{0,1,2\}$, and is equal to $=\{0\}$ otherwise. Moreover $$ \{ E_0(e) , E_2 (e), E\}=\{0\} =  \{ E_2(e) , E_0 (e), E\}.$$

A tripotent $e$ in $E$ is called \emph{complete}, or \emph{unitary}, or \emph{minimal} if $E_0(E)=0$, or $E_2(e)=E$, or $E_2(e)=\CC e \neq \{0\},$ respectively. We say that $e$ has \emph{finite rank} if it can be written a sum of finitely many mutually orthogonal minimal tripotents.\smallskip

Every hermitian element in a C$^*$-algebra defines a commutative C$^*$-algebra and that allows us to define a local functional calculus. However, for non-normal elements the spectral resolutions and the functional calculus is hopeless. There is a certain advantage in the setting of JB$^*$-triples. Accordingly to the standard notation, given an element $x$ in a JB$^*$-triple $E$, we shall write $x^{[1]} := x$, $x^{[3]} := \J xxx$, and $x^{[2n+1]} := \J xx{x^{[2n-1]}},$ $(n\in \NN)$, while the symbol $E_x$ will stand for the JB$^*$-subtriple generated by the element $x$, that is, the norm closure of the linear span of all odd powers ${x^{[2n-1]}}$ ($n\in \mathbb{N}$) of $x$. It is known that
$E_x$ is JB$^*$-triple isomorphic (and hence isometric) to $C_0 (\Omega)$ for some locally compact Hausdorff space $\Omega$
contained in $(0,\|x\|],$ such that $\Omega\cup \{0\}$ is compact and $\|x\|\in \Omega$, where $C_0 (\Omega)$ denotes the Banach space of all
complex-valued continuous functions vanishing at $0.$ It is also known that there exists a triple isomorphism $\Psi$ from $E_x$ onto $C_{0}(\Omega),$ satisfying $\Psi (x) (t) = t$ $(t\in \Omega)$ (cf. \cite[Corollary 1.15]{Ka83}).\smallskip

Elements $a,b$ in a JB$^*$-triple $E$ are said to be \emph{orthogonal} (written $a\perp b$) if $L(a,b) =0$. It is known that $a\perp b$ $\Leftrightarrow$ $\J aab =0$ $\Leftrightarrow$ $\{b,b,a\}=0$ $\Leftrightarrow$ $b\perp a;$ $\Leftrightarrow$ $E_a \perp E_b$ (see, for example, \cite[Lemma 1]{BurFerGarMarPe}). Let $e$ be a tripotent in $E$. It follows from the Peirce arithmetic that $a\perp b$ for every $a\in E_2(e)$ and every $b\in E_0(e)$. It is known that $a\perp b$ in $E$ implies that $\|\lambda a + \mu b\| = \max\{\|\lambda a\| ,\|\mu b\|\}$ (compare \cite[Lemma\  1.3(a)]{FriRu85}).\smallskip

The relation ``being orthogonal'' is applied to define a partial order on the set of tripotents of a JB$^*$-triple $E$ defined by the following: given $u,e\in \mathcal{U} (E)$, we write $u \leq e$ if $e-u$ is a tripotent in $E$ and $e-u \perp u$. The following fact will be needed later. Let $e$ and $u$ be tripotents in a JB$^*$-triple $E$, then \begin{equation}\label{eq orthogonality via pm} u\perp e \Leftrightarrow u\pm e \hbox{ are a tripotents}
\end{equation} (see \cite[Lemma 3.6]{IsKaRo95}).\smallskip

When $a$ is a norm one element in a JBW$^*$-triple $W$, the sequence $(a^{[2n -1]})$ converges in the weak$^*$ topology of $W$ to a tripotent (called the \hyphenation{support}\emph{support} \emph{tripotent} of $a$) $s(a)$ in $W$ (compare \cite[Lemma 3.3]{EdRutt88} or \cite[page 130]{EdFerHosPe2010}). It is known that $a=s(a) + P_0 (s(a)) (a)$. For a norm one element $a$ in a JB$^*$-triple $E$, $s(a)$ will denote the support tripotent of $a$ in $E^{**}$.\smallskip

Following \cite{EdRu96}, a tripotent $e$ in the second dual of a JB$^*$-triple $E$ is said to be \emph{compact-$G_{\delta}$} if there exists a norm one element $a$ in $E$ satisfying $s(a) =e$. A tripotent $e$ in $E^{**}$ is called \emph{compact} if $e=0$ or it is the infimum of a decreasing net of compact-$G_{\delta}$ tripotents in $E^{**}$. The set of all compact tripotents in $E^{**}$ will be denoted by $\mathcal{U}_{c} (E^{**})$.\smallskip

According to \cite{FerPe06}, a tripotent $e\in E^{**}$ is said to be \emph{bounded} if there exists a norm one element $a$ in $E$ such that $P_2 (e) (a) =e$. It is known that, in these circumstances, $P_1 (e) (a) = 0$. We shall write $e\leq_{T} a$ when $a$ satisfies the above conditions. The relation
$\leq_{_T}$ is consistent with the natural partial order on the set of tripotents, that is, for any two tripotents $e$ and $u$ we
have $e\leq u$ if and only if $e\leq_{_T} u.$ Following the same reference, a tripotent $e$ in $E^{**}$ satisfying that $E^{**}_0 (e)\cap E$ is weak$^*$ dense in $E^{**}_0(e)$ is called \emph{closed} relative to $E$.\smallskip

The following characterization of compact tripotents in the second dual of a JB$^*$-triple has been borrowed from \cite[Theorem 2.6]{FerPe06} (see also \cite[Theorem 3.2]{FerPe10b}), and will be applied later.

\begin{theorem}\cite[Theorem 2.6]{FerPe06}\label{th FerPer compact trip} Let $E$ be a JB$^*$-triple. Then a tripotent $e$ in $E^{**}$ is compact if and only if $e$ is closed and bounded.$\hfill\Box$
\end{theorem}

After having introduced the necessary concepts, the norm closed faces of the closed unit ball of a JB$^*$-triple $E$ can be characterized in terms of the compact tripotents in $E^{**}$ via the next theorem, which due to Edwards, Fern{\'a}ndez-Polo, Hoskin and the first author of this note (see \cite{EdFerHosPe2010}).

\begin{theorem}\label{thm norm closed faces} Let $E$ be a JB$^*$-triple, and let $F$ be a non-empty norm closed face of
the unit ball $E_1$ in $E$. Then, there exists uniquely a compact tripotent
$u$ in $E^{**}$ such that
\[
F = (u + E_0^{**}(u)_1) \cap E = (\{u\}_{\prime})_{\prime},
\]
where $E_0^{**}(u)_1$ is the unit ball in the Peirce-zero space $E^{**}_0(u)$
in $E^{**}$ and $(\{u\}_{\prime})_{\prime}$ is the norm-semi-exposed face of $E_1$
corresponding to $u$ {\rm(}as defined in \eqref{1.1}{\rm)}. Furthermore,
the mapping $u \mapsto (\{u\}_{\prime})_{\prime}$ is an anti-order isomorphism from
$\tilde{\mathcal{U}}_c( E^{**})$ onto the complete lattice $\mathcal{F}_n(E_1)$ of norm closed
faces of $E_1$.$\hfill\Box$
\end{theorem}

Let $E$ be a JB$^*$-triple. It is not obvious but every minimal tripotent, actually every finite rank tripotent in $E^{**}$ is compact (see \cite[Theorem 3.4]{BuFerMarPe}). It is known that the minimal elements of the set $\mathcal{U}_{c} (E^{**})$ are precisely the minimal tripotents of $E^{**}$  (compare the comments before \cite[Corollary 3.5]{BuFerMarPe}). We shall see in the next subsection that finite rank tripotents can be also applied to characterize compact JB$^*$-triples.

\subsection{Weakly compact JB$^*$-triples}

The facial structure of the closed unit ball is much simpler in the case of a compact JB$^*$-triple. For this reason we briefly survey the basic notions of compact JB$^*$-triples. Motivated by the studies on compact C$^*$-algebras published by Alexander and Ylinen (see \cite{Alex,Yli}), Bunce and Chu introduced and studied in \cite{BuChu} the notions of compact and weakly compact JB$^*$-triples.

An element $x$ in a JB$^*$-triple $E$ is called \emph{compact} (respectively, \emph{weakly compact}) if the operator $Q(x):E\rightarrow E$, $z\mapsto \{x,z,x\}$ is compact (respectively,  weakly compact). The JB$^*$-triple $E$ is \emph{compact} (respectively, \emph{weakly compact}) if every element in $E$ is compact (respectively, weakly compact).\smallskip

The connections between compact JB$^*$-triples and finite rank tripotents are very useful to obtain a concrete representation of every compact JB$^*$-triple. Following the notation in \cite{BuChu}, the Banach subspace of a JB$^*$-triple $E$ generated by all its minimal tripotents will be denoted by $K(E)$. It is known that $K(E)$ is a (norm closed) triple ideal of $E$ which coincides with the set of all weakly compact elements in $E$ (see \cite[Proposition 4.7]{BuChu}). In order to have a more concrete representation, we recall that for each Cartan factor $C$ the elementary JB$^*$-triple associated with $C$ is precisely $K(C).$ That is, there are six different types of elementary JB$^*$-triples, denoted by $K_i$ $(i = 1, ... , 6)$, which defined as follows: $K_1 = K (H, H')$ (the compact operators between two complex Hilbert spaces $H$ and $H'$); $K_i = C_i \cap K(H)$ for $i = 2 , 3$, and $K_i = C_i$ for $i = 4, 5,
6$. The following structure theorem was established by Bunce and Chu in \cite{BuChu}.

\begin{theorem}\label{thm compact triples BuncChu}\cite[Lemma 3.3 and Theorem 3.4]{BuChu} Let $E$ be a JB$^*$ triple. Then $E$ is weakly compact if
and only if one of the following statement holds: \begin{enumerate}[$a)$] \item $K(E^{**})=K(E)$; \item $K(E)=E;$
\item $E$ is a $c_0$-sum of elementary JB$^*$-triples.$\hfill\Box$
\end{enumerate}
\end{theorem}

Let $E$ be a reflexive JB$^*$-triple. Clearly, every compact tripotent in $E^{**}$ lies in $E$. The same conclusion holds when $E$ is a weakly compact JB$^*$-triple. More concretely, suppose $E$ is a weakly compact JB$^*$-triple and $e$ is a compact tripotent in $E^{**}$. By Theorem \ref{th FerPer compact trip} there exists a norm one element $a$ in $E$ such that $e\leq_{T} a$. An application of Theorem \ref{thm compact triples BuncChu} assures that $e$ must be a finite rank tripotent in $E$.

\begin{corollary}\label{c compact tripotents in the bidual of a weakly compact JBtriple} Let $E$ be a weakly compact JB$^*$-triple. Then every compact tripotent in $E^{**}$ is a finite rank tripotent in $E$.$\hfill\Box$
\end{corollary}

\section{Solution to Tingley's problem for weakly compact JB$^*$-triples}

This section is devoted to obtain a complete solution to Tingley's problem in the case of surjective isometries between the unit spheres of arbitrary weakly compact JB$^*$-triples and several consequences in the setting of C$^*$-algebras and operator spaces. We begin with a key proposition which improves the conclusion of Lemma \ref{l 1.1}.\smallskip

Henceforth, given a vector $x_0$ in a Banach space $X$, the translation with respect to $x_0$ will be denoted by $\mathcal{T}_{x_0}$.

\begin{proposition}\label{p surjective isometries between the spheres preserve norm closed faces} Let $E$ and $B$ be weakly compact JB$^*$-triples, and  suppose that $f: S(E) \to S(B)$ is a surjective isometry. Then the following statements hold:\begin{enumerate}[$(a)$] \item For each minimal tripotent $e_1$ in $E$ there exists a unique minimal tripotent $u_1$ in $B$ such that $f( (e_1 + E_0^{**}(e_1)_1) \cap E )= (u_1 + B_0^{**}(u_1)_1) \cap B$;
\item The restriction of $f$ to each   maximal proper face of $E_1$ is an affine function;
\item For each minimal tripotent $e_1$ in $E$ there exists a unique minimal tripotent $u_1$ in $B$ such that $f(e_1 )= u_1$;
\item $f$ maps norm closed proper faces of $E_1$ to norm closed faces of $B_1$.
\end{enumerate}
\end{proposition}

\begin{proof} Let $M\subseteq S(E)$ be a maximal proper (norm closed) face of $E_1$. Lemma \ref{l 1.1} implies that $f(M)$ is maximal proper (norm closed) face of $E_1$. Combining Theorem \ref{thm norm closed faces} and subsequent comments with Corollary \ref{c compact tripotents in the bidual of a weakly compact JBtriple}, we deduce the existence of minimal tripotents $e_1\in E$ and $u_1\in B$ such that $M = (e_1 + E_0^{**}(e_1)_1) \cap E $ and $f(M) = (u_1 + B_0^{**}(u_1)_1) \cap B$. Since by Theorem \ref{thm norm closed faces} and Corollary \ref{c compact tripotents in the bidual of a weakly compact JBtriple} minimal tripotents in $E$ and $B$ are in one-to-one correspondence with the maximal proper (norm closed) face of $E_1$ and $B_1$, respectively, the statement in $(a)$ follows from the above arguments. \smallskip

We also know that, under the above assumptions, $E_0^{**}(e_1)\cap E=E_0(e_1)$ and $ B_0^{**}(u_1) \cap B=B_0(u_1)$ are weak$^*$-dense (norm-closed) subspaces of $E_0^{**}(w_1)$ and $ B_0^{**}(u_1)_1$, respectively, whose closed unit balls are precisely $E_0(e_1)_1$ and $ B_0(u_1)_1$, respectively. The mapping $f_{e_1} = \mathcal{T}_{u_1}^{-1}|_{f(M)} \circ f|M \circ \mathcal{T}_{e_1}|_{E_0(e_1)_1 }$ is a surjective isometry from $E_0(e_1)_1 $ onto $ B_0(u_1)_1$. Mankiewicz's theorem (see \cite{Mank1972}) assures the existence of a surjective real linear isometry $T_{e_1} : E_0(e_1)\to  B_0(u_1)$ such that $f_{e_1} = T_{e_1}|_{S(E_0(e_1))}$. Since translations and linear isometries are affine functions, the identity $f|_{M} =  \mathcal{T}_{e_1}^{-1} |_{M} \circ T_{e_1}|_{S(E_0(e_1))} \circ \mathcal{T}_{u_1}|_{B_0(u_1)_1 B}$ proves that $f|_{M}$ is an affine function, which proves $(b)$. We also know that $f(e_1) = u_1$, which gives $(c)$.\smallskip

We shall finally prove $(d)$. Let $F$ be a norm closed proper face of $E_1$. As before, an appropriate combination of Theorem \ref{thm norm closed faces} and Corollary \ref{c compact tripotents in the bidual of a weakly compact JBtriple} implies the existence of a finite rank tripotent $e$ in $E$ such that $F = (e + E_0^{**}(e)_1) \cap E =\{e\}_{\prime\prime}$. Take a minimal tripotent $e_1$ such that $e_1\leq e$. Since $F\subseteq \{e_1\}_{\prime\prime}$, the latter is a maximal proper face of $E_1$, and, by $(b)$, $f|_{\{e\}_{\prime\prime}}$ is affine, we deduce that $f(F)$ is a convex subset of $S(B)$. Applying that $f$ is an isometry we can easily see that $f(F)$ is closed. Suppose that $t a + (1-t) b = f(c)$, where $a,b\in S(B)$, $c\in F$ and $t\in (0,1)$. Pick a maximal proper face $M_1\subset S(B)$ such that $f(\{e\}_{\prime\prime}) =M_1$. Since $M_1$ is a norm closed face of $B_1$, it follows that $a,b\in M_1$. Applying statement $(b)$ to $f^{-1}|_{M_1}$ we get $t f^{-1} (a) + (1-t) f^{-1} (b) = f^{-1} (t a + (1-t) b)  = c\in F,$ and thus $a,b\in f(F)$, because $F$ is a face. This shows that $f(F)$ is a norm closed proper face of $B_1$.
\end{proof}

The next result, whose proof is essentially based on the arguments given in the proof previous proposition, goes deeper in the conclusions given above.

\begin{proposition}\label{p surjective isometries between the spheres preserve finite rank tripotents} Let $E$ and $B$ be weakly compact JB$^*$-triples, and  suppose that $f: S(E) \to S(B)$ is a surjective isometry. Then the following statements hold:\begin{enumerate}[$(a)$] \item For each finite rank tripotent $e$ in $E$ there exists a unique finite rank tripotent $u$ in $B$ such that $f( (e + E_0^{**}(e)_1) \cap E )= (u + B_0^{**}(u)_1) \cap B$;
\item The restriction of $f$ to each norm closed face of $E_1$ is an affine function;
\item For each finite rank tripotent $e$ in $E$ there exists a unique finite rank tripotent $u$ in $B$ and a surjective real linear isometry $T_e : E_0(e) \to B_0(u)$ such that $$f(e +x )= u + T_e (x),$$ for every $x\in E_0(e)_1 $;
\item For each finite rank tripotent $e$ in $E$ there exists a unique finite rank tripotent $u$ in $B$ such that $f(e )= u$.
\end{enumerate}
\end{proposition}

\begin{proof} $(a)$ Follows from Theorem \ref{thm norm closed faces}, Corollary \ref{c compact tripotents in the bidual of a weakly compact JBtriple} and Proposition \ref{p surjective isometries between the spheres preserve norm closed faces}$(d)$.\smallskip

Let $e$ be a finite rank tripotent in $E$. By $(a)$ for each finite rank tripotent $e$ in $E$ there exists a unique finite rank tripotent $u$ in $B$ such that $$f( (e + E_0^{**}(e)_1) \cap E )= (u + B_0^{**}(u)_1) \cap B.$$ Arguing as in the proof of Proposition \ref{p surjective isometries between the spheres preserve norm closed faces}, the mapping $f_{e} = \mathcal{T}_{u}^{-1}|_{f(\{e\}_{\prime\prime})} \circ f|_{\{e\}_{\prime\prime}} \circ \mathcal{T}_{e}|_{E_0(e)_1}$ is a surjective isometry from $E_0(e)_1$ onto $ B_0(u)_1$. By Mankiewicz's theorem (see \cite{Mank1972}) there exists a surjective real linear isometry $T_{e} : E_0(e)\to  B_0(u)$ such that $f_{e} = T_{e}|_{S(E_0(e))}$. Therefore $f|_{\{e\}_{\prime\prime}} =  \mathcal{T}_{e}^{-1} |_{\{e\}_{\prime\prime}} \circ T_{e}|_{S(E_0(e))} \circ \mathcal{T}_{u}|_{B_0(u)_1}$ is an affine function, which proves $(b)$. We further know that $$f(e +x )= u + T_e (x),$$ for every $x\in E_0(e)_1 $ and $f(e) = u$, which gives $(c)$ and $(d)$.\end{proof}

We recall that the rank of a finite rank tripotent $e$ in a JB$^*$-triple $E$ is the unique natural $k$ satisfying that $e$ writes as a sum of $k$ mutually orthogonal minimal tripotents in $E$. A subset $S$ of a JB$^*$-triple $E$ is called \emph{orthogonal} if $0 \notin S$ and $x \perp y$ for every $x\neq y$ in $S$. The minimal cardinal number $r$ satisfying $card(S) \leq r$ for every orthogonal subset $S \subseteq E$ is called the \emph{rank} of $E$.\smallskip

We are now in position to prove that a surjective isometry between the unit spheres of two weakly compact JB$^*$-triples maps tripotents of rank $k$ to tripotents of rank $k$.  We begin with some technical lemmas. The first one is a geometric version of \eqref{eq orthogonality via pm} (see \cite[Lemma 3.6]{IsKaRo95}), a version of which was considered in \cite{FerMarPe2012}. \smallskip

We recall that, given a subset $S$ of the closed unit ball of a Banach space $X$, the set of all
contractive perturbations of $S$, $\hbox{\rm cp}(S)$, is defined by $$  \hbox{\rm cp}(S) =\left\{ x\in X : \|x \pm s \| \leq 1, \hbox{for every $s\in S$}\right\} \hbox{  (see \cite{FerMarPe2012}).}$$ 

\begin{lemma}\label{l two finite rank tripotents at distance 1 are orthogonal}\cite{FerMarPe2012} Let $e$ be a tripotent in a JB$^*$-triple $E$. Let $x$ be an element in $E_1$ satisfying $\|e\pm x\| = 1$. Then $x\perp  e$.
\end{lemma}

\begin{proof} The identity $(6)$ in \cite{FerMarPe2012} (see also \cite[Corollary 4.3]{EdHu}) shows that $$\{e\}^{\perp}\cap E_1 = \hbox{\rm cp} (\{e\}),$$ where $\{e\}^{\perp}= \{x\in E : e\perp x\}$. An element $x$ in the hypothesis of the Lemma belongs to $\{e\}^{\perp}\cap E_1$ and hence $x\perp e$. \end{proof}

\begin{lemma}\label{l two finite rank tripotents at distance 2} Let $e$ and $w$ be finite rank tripotents in a weakly compact JB$^*$-triple $E$ not containing direct summands of rank smaller or equal than 3. Suppose that $e$ is minimal and $\|e-w\| = 2$. Then $w= -e + P_0(e) (w)$.
\end{lemma}

\begin{proof} By the structure theory of weakly compact JB$^*$-triples (see Theorem \ref{thm compact triples BuncChu}), there is no loss of generality in assuming that $E$ is an elementary JB$^*$-triple of rank bigger or equal than 4.\smallskip

We assume first that $E=K(H_1,H_2)$, where $H_1$ and $H_2$ are complex Hilbert spaces (i.e. $E$ is an elementary JB$^*$-triple of type $K_1$). It is well known that in this case, $e = \xi \otimes \eta$ and  $w = \sum_{j=1}^{m} \zeta_j \otimes \vartheta_j$, where $\{\zeta_1,\ldots,\zeta_{m}\}$ and $\{\vartheta_1,\ldots,\vartheta_{m}\}$ are orthonormal systems in $H_1$ and $H_2$, respectively. The element $w-e$ is a finite rank operator with $\|e-w\|=2$. Then there exists a norm one element $h\in H_1$ such that $\| (e-w) (h)\| =2$. If $|\langle h, \eta \rangle|<1$ then $$2 = \| (e-w) (h)\| \leq \|w(h)\| + \left\| \langle h, \eta \rangle \xi \right\| < 2,$$ which is impossible. So, $|\langle h, \eta \rangle| = 1$, and hence $h = \langle h, \eta \rangle \eta$. Similarly, if $|\langle h, \vartheta_j \rangle|<1$, for every $1\leq j \leq m$, we would have $2=\| (e-w) (h)\| \leq \| e (h)\| + \| w (h)\| < 2$, which is impossible. Therefore there exists $j_1\in \{1,\ldots, m \}$ such that $|\langle h, \vartheta_{j_1} \rangle|=1$ and $h = \langle h, \vartheta_{j_1} \rangle \vartheta_{j_1}$. In particular $\eta =\lambda \vartheta_{j_1}$ for a suitable $\lambda$ in $\mathbb{C}$ with $|\lambda |=1$.\smallskip

Since $\|e^*-w^*\| = 2$ it follows from the above arguments the existence of $j_2\in \{1,\ldots, m \}$ such that $\xi =\mu \zeta_{j_2}$ for a suitable $\mu$ in $\mathbb{C}$ with $|\mu |=1$. Therefore $e = \overline{\lambda} \mu \zeta_{j_2}\otimes \vartheta_{j_1}$, and the condition $$2 =\| (e-w) (h)\| = \|\overline{\lambda} \mu \langle h, \vartheta_{j_1} \rangle \zeta_{j_2} - \langle h, \vartheta_{j_1} \rangle \zeta_{j_1}\|$$ implies that $j_1=j_2$, $\overline{\lambda} \mu=1$, and hence $w =- e +P_0(e) (w)$. The cases in which $E$ is of type $K_2$ or $K_3$ follow by similar arguments.
\end{proof}

Since every compact C$^*$-algebra can be written as a $c_0$ sum of C$^*$-algebras of the form $K(H_i)$, where each $H_i$ is a complex Hilbert space, the proof of the previous lemma actually shows the following statement.

\begin{lemma}\label{l two finite rank tripotents at distance 2 in KH} Let $e$ and $w$ be finite rank tripotents in a compact C$^*$-algebra $A$. Suppose that $e$ is minimal and $\|e-w\| = 2$. Then $w= -e + P_0(e) (w)$. $\hfill\Box$
\end{lemma}

The following result is a generalization of Tingley's theorem in the setting of weakly compact JB$^*$-triples.

\begin{theorem}\label{t Tingley antipodes for finite rank} Let $f: S(E) \to S(B)$ be a surjective isometry between the unit spheres of two weakly compact JB$^*$-triples. We assume that $E$ and $B$ do not contain direct summands of rank smaller or equal than 3. Suppose $e$ is a finite rank tripotent in $E$. Then $f(-e)=-f(e)$. Furthermore, if $e_1,\ldots,e_m$ are mutually orthogonal minimal tripotents in $E$, then $$ f(e_1+\ldots+e_m) = f(e_1)+\ldots+f(e_m).$$
\end{theorem}

\begin{proof} We shall first prove that \begin{equation}\label{eq f preserves antipodes minimal trip} f(-e_1) = -f(e_1), \hbox{ for every minimal tripotent } e_1\in E.
\end{equation} Indeed, by Proposition \ref{p surjective isometries between the spheres preserve norm closed faces} $v_1 = f(e_1)$ and $w_1 = f(-e_1)$ are minimal tripotents in $B$. By the assumptions on $f$ we have $$ \| v_1 - w_1 \| = \| f (e_1) - f (-e_1)\|= \| e_1  + e_1\| =2,$$ which, via Lemma \ref{l two finite rank tripotents at distance 2}, proves that $w_1 = -v_1 + P_0(v_1) (w_1)$. However, $w_1$ being minimal implies that $w_1=- v_1$, which proves \eqref{eq f preserves antipodes minimal trip}.\smallskip

Let $e_1,\ldots,e_m$ be mutually orthogonal minimal tripotents in $E$. Proposition \ref{p surjective isometries between the spheres preserve norm closed faces} assures that $v_j =f(e_j)$ is a minimal tripotent for every $1\leq j\leq m$. We also know from \eqref{eq f preserves antipodes minimal trip} that $f(-e_j) = -f(e_j)$, for every such $j$.\smallskip

We claim that \begin{equation}\label{eq vj are mutually orthogonal} v_j \perp v_k, \hbox{ for every } j\neq k.
\end{equation} To see this, for each $j\neq k$, we observe that $$\| v_j \pm v_k \| = \| f(e_j) \pm f(e_k) \| = \| f(e_j) -  f(\pm e_k) \| = \| e_j \pm e_k\| =1,$$ and hence Lemma \ref{l two finite rank tripotents at distance 1 are orthogonal} implies that $v_j \perp v_k$.\smallskip

Let $w=f(-e_1-\ldots-e_m)$ and $v=f(e_1+\ldots+e_m)$. We shall show that $w=- v$. Proposition \ref{p surjective isometries between the spheres preserve finite rank tripotents} implies that $w$ and $v$ are finite rank tripotents. Fix $1\leq j\leq m$. It follows from the hypothesis that $$\|w-v_j \| =\| -e_1-\ldots-e_m - e_j\|=2.$$ An application of Lemma \ref{l two finite rank tripotents at distance 2} shows that $w = -v_j + P_0(v_j) (w)$, for every $1\leq j\leq m$. Having in mind that, by \eqref{eq vj are mutually orthogonal}, $v_1,\ldots, v_m$ are mutually orthogonal, we can easily deduce that $$f(-e_1-\ldots-e_m)=w = -v_1 -\ldots - v_m + P_0(v_1+\ldots+ v_m) (w) $$ $$= -v + P_0(v_1+\ldots+ v_m) (w) = -f(e_1+\ldots+e_m) + P_0(v_1+\ldots+ v_m) (w).$$ This shows that $f(-e_1-\ldots-e_m) = w$ has rank greater or equal than $m$. If $w$ had rank strictly bigger than $m$, it would follow from the above arguments applied to $f^{-1}$ that $f^{-1} (w) =-e_1-\ldots-e_m $ had rank strictly bigger than $m$, which is impossible. Therefore $P_0(v_1+\ldots+ v_m) (w)=0$ and hence $$f(-e_1-\ldots-e_m) = -f(e_1)-\ldots-f(e_m) .$$ The above arguments also show that $f(e_1+\ldots+e_m) = f(e_1)+\ldots+f(e_m),$ and thus $f(-e_1-\ldots-e_m) = -f(e_1+\ldots +e_m) .$
\end{proof}

When in the proof of Theorem \ref{t Tingley antipodes for finite rank}, Lemma \ref{l two finite rank tripotents at distance 2} is replaced with Lemma \ref{l two finite rank tripotents at distance 2 in KH} we obtain the following result.

\begin{theorem}\label{t Tingley antipodes for finite rank compact C*-algebras} Let $f: S(A) \to S(B)$ be a surjective isometry between the unit spheres of two compact C$^*$-algebras. Suppose $e$ is a finite rank tripotent in $E$. Then $f(-e)=-f(e)$. Furthermore, if $e_1,\ldots,e_m$ are mutually orthogonal minimal tripotents in $E$, then $$ f(e_1+\ldots+e_m) = f(e_1)+\ldots+f(e_m).$$ $\hfill\Box$
\end{theorem}

The next result is a direct consequence of Theorems \ref{t Tingley antipodes for finite rank} and \ref{t Tingley antipodes for finite rank compact C*-algebras}.

\begin{corollary}\label{c rank preservation bewteen elementary} Let $f: S(E) \to S(B)$ be a surjective isometry between the unit spheres of two weakly compact JB$^*$-triples. We assume that $E$ and $B$ do not contain direct summands of rank smaller or equal than 3. Then the following statements hold:
\begin{enumerate}[$(a)$]\item If $e_1$ and $e_2$ are two orthogonal finite rank tripotents in $E$. Then $f(e_1)$ and $f(e_2)$ are two orthogonal finite rank tripotents in $B$ with $f(e_1 + e_2 ) = f(e_1) + f(e_2)$.
\item If $e$ is a rank $k$ tripotent in $E$ then $f(e)$ is a rank $k$ tripotent in $B$.
\item The rank of $E$ and the rank of $B$ coincide.
\end{enumerate}
Furthermore, the same statements hold when $E$ and $B$ are compact C$^*$-algebras.
\end{corollary}

Our next proposition plays a fundamental role in our results.

\begin{proposition}\label{p identity principle} Let $f: S(E) \to S(B)$ be a surjective isometry between the unit spheres of two weakly compact JB$^*$-triples. Let $A$ be a JB$^*$-subtriple of $E$, and suppose that $T : A\to B$ is a bounded linear operator such that $T(u) =f(u)$ for every finite rank tripotent $u\in A$. Then $f(x) =T(x),$ for every $x\in S(A)$.
\end{proposition}

\begin{proof} Let $x$ be an element in $S(A)$. Since every JB$^*$-subtriple of a weakly compact JB$^*$-triple is weakly compact (see \cite[Lemma 3.2]{BuChu}), it follows that $A$ is weakly compact. Therefore, we can assure that $x$ can be approximated in norm by an element of the form $\displaystyle z = \sum_{j=1}^{m} \lambda_j e_j$, where $e_1,\ldots,e_m$ are mutually orthogonal minimal tripotents in $A$ and $0<\lambda_j \leq \|x\|=1$ for every $j$ (see \cite[Remark 4.6]{BuChu}). The element $z$ can be chosen with the condition that $\|z\|=1$ (that is, $\lambda_j = 1$ for some $j\in \{1,\ldots, m\}$). By the norm-density of this type of elements and the continuity of $f$ and $T$, the desired conclusion follows as soon as we prove that $f(z) = T (z)$, for every $z$ as above.\smallskip

Let us write such an element $z$ in the form $$\displaystyle z = u_1+\ldots+u_{k} + \sum_{j=k+1}^{m} \lambda_j u_j = u+ \sum_{j=k+1}^{m} \lambda_j u_j,$$ where $u_1,\ldots, u_m$ are mutually orthogonal minimal tripotents in $A$ (and in $E$), $u= u_1+\ldots+ u_k$, and $0<\lambda_j<1$. Let $\mathcal{C}$ denote the JB$^*$-subtriple of $A$ generated by $\{u_{k+1},\ldots, u_m\}$. Clearly, $\mathcal{C}$ is JB$^*$-triple isomorphic to a unital C$^*$-algebra of dimension $m-k$ and $\displaystyle \sum_{j=k+1}^{m} \lambda_j u_j\in \mathcal{C}$ with $\displaystyle \left\|\sum_{j=k+1}^{m} \lambda_j u_j \right\| <1$. By the Russo-Dye theorem (see \cite{RuDye}), or by the strengthened version proved by Kadison and Pedersen in \cite{KadPed}, the element $\displaystyle \sum_{j=k+1}^{m} \lambda_j u_j$ is the mean of a finite number of unitary elements in $\mathcal{C}$, that is $$\sum_{j=k+1}^{m} \lambda_j u_j = \frac{1}{m_1} (w_1+\ldots+w_{m_1}),$$ where $w_1,\ldots,w_{m_1}$ are unitary elements in the C$^*$-algebra $\mathcal{C}$. It is not hard to see that, defining $\displaystyle \widetilde{w}_{j} = u + w_j,$ $j\in \{1,\ldots,m_1\}$, it follows that $\widetilde{w}_{1},\ldots,\widetilde{w}_{m_1}$ are tripotents in $A$ (and in $E$) and \begin{equation}\label{eq finite convex combination1} z = \frac{1}{m_1} (\widetilde{w}_1+\ldots+\widetilde{w}_{m_1}).
\end{equation} Clearly $z,\widetilde{w}_{1},\ldots,\widetilde{w}_{m_1}$ lie in a proper norm closed face of $E_1$, so applying Proposition \ref{p surjective isometries between the spheres preserve finite rank tripotents}$(b)$ and the hypothesis we get $$f(z) = \frac{1}{m_1} (f(\widetilde{w}_1) +\ldots+f(\widetilde{w}_{m_1})) = \frac{1}{m_1} (T(\widetilde{w}_1) +\ldots+T(\widetilde{w}_{m_1})) = T(z),$$ which concludes the proof.
\end{proof}

\begin{corollary}\label{cor Te coincides with f on the orthogonal of e} Let $f: S(E) \to S(B)$ be a surjective isometry between the unit spheres of two weakly compact JB$^*$-triples not containing direct summands of rank smaller or equal than 3. Let $e$ be a finite rank tripotent in $E$, $u= f(e)$, and let $T_e : E_0(e) \to B_0(u)$ be the surjective real linear isometry given by Proposition \ref{p surjective isometries between the spheres preserve finite rank tripotents}$(c)$. Then $f(x) = T_e (x)$ for all $x\in S(E_0(e))$. The conclusion also holds when $E$ and $B$ are compact C$^*$-algebras.
\end{corollary}

\begin{proof} Let $e_1$ be a rank-one tripotent in $E_0(e)$. By Corollary \ref{c rank preservation bewteen elementary} we have $f(e\pm e_1) = f(e) \pm f(e_1)$, and from Proposition \ref{p surjective isometries between the spheres preserve finite rank tripotents}$(c)$ we also known that $$ f(e) \pm f(e_1)= f (e \pm e_1 ) = f(e) + T_e (\pm e_1),$$ which implies that $T_e (e_1) = f(e_1)$. A new application of Corollary \ref{c rank preservation bewteen elementary} proves that \begin{equation}\label{eq Te and f coincide on finite ranks} T_e (u) = f(u),
\end{equation} for every finite rank tripotent $u$ in $E_0(e)$. Since $E_0(e)$ is a JB$^*$-subtriple of $E$, the desired conclusion follows from \eqref{eq Te and f coincide on finite ranks} and Proposition \ref{p identity principle}.
\end{proof}

We shall need some additional properties of the real linear isometries $T_e$ given by Proposition \ref{p surjective isometries between the spheres preserve finite rank tripotents}$(c)$.

\begin{lemma}\label{l Te and Tv coincide on the intersection} $f: S(E) \to S(B)$ be a surjective isometry between the unit spheres of two weakly compact JB$^*$-triples not containing direct summands of rank smaller or equal than 3, or between two compact C$^*$-algebras. Let $e_1$ and $e_2$ be two orthogonal finite rank tripotents in $E$, and let $T_{e_1}$ and $T_{e_2}$ be the maps given by Proposition \ref{p surjective isometries between the spheres preserve finite rank tripotents}$(c)$. Then $$\hbox{$T_{e_{1}} (x) = T_{e_{2}} (x)$ for all $x\in E_0(e_{1}) \cap E_0(e_{2})$.}$$
\end{lemma}

\begin{proof} Let $x$ be an element in $E_0 (e_{1}) \cap E_0(e_{2})$ with $\|x\|\leq 1$. Let $u_j = f(e_j)$. We deduce, via Proposition \ref{p surjective isometries between the spheres preserve finite rank tripotents}$(c)$, that $$u_2 + T_{e_{2}} (e_{1}+x)= f(e_{1} + e_{2} + x) = u_1 + T_{e_{1}} (e_{2}+ x),$$ and by Corollary \ref{c rank preservation bewteen elementary} $$u_1+u_2= f(e_{1}) +f (e_{2}) = f(e_{1} +e_{2}) = u_1 + T_{e_{1}} (e_{2}) = u_2 + T_{e_{2}} (e_{1}).$$ Therefore, $u_1 = T_{e_{2}} (e_{1}),$ $T_{e_{1}} (e_{2}) = u_2$, $T_{e_{1}} (x) = T_{e_{2}} (x)$, which proves the statement.
\end{proof}

We can state now a first answer to Tingley's problem in the case of weakly compact JB$^*$-triples which are expressed as a sum of at least two elementary JB$^*$-triples.

\begin{theorem}\label{thm Tyngley co sums with more than one element} Let $f: S(E) \to S(B)$ be a surjective isometry between the unit spheres of two weakly compact JB$^*$-triples not containing direct summands of rank smaller or equal than 3 {\rm(}respectively, between two compact C$^*$-algebras{\rm)}. Suppose that $\displaystyle E=\oplus^{c_0}_{j\in J} K_j$, where $\sharp J \geq 2$, and every $K_j$ is an elementary JB$^*$-triple {\rm(}respectively, each $K_j$ coincides with $K(H_j)$ for a suitable complex Hilbert space $H_j${\rm)}. Then there exists a surjective real linear isometry $T: E\to B$ satisfying $T|_{S(E)} = f$.
\end{theorem}

\begin{proof} Since  $\sharp J \geq 2$, we can pick two different subindexes $j_1$ and $j_2$ in $J$. Let $e_{_{j_1}}\in K_{j_1}$ and $e_{_{j_2}}\in K_{j_2}$ be finite rank tripotents, and let  $T_{e_j} : E_0(e_j) \to B_0(u_j)$ be the surjective real linear isometry given by Proposition \ref{p surjective isometries between the spheres preserve finite rank tripotents}$(c)$, where $u_j = f(e_j)$. By Corollary \ref{cor Te coincides with f on the orthogonal of e}  we know that $f(x) = T_{e_j} (x)$ for all $x\in S(E_0(e_j))$.\smallskip

It is easy to see that \begin{equation}\label{eq decomposition of E orthogonals} E = K_{j_1}\oplus K_{j_{2}}\oplus (\oplus^{c_0}_{j\neq j_1,j_2} K_j), \end{equation} with $\displaystyle\oplus^{c_0}_{j\neq j_1,j_2} K_j \subseteq E_0(e_{_{j_1}}) \cap E_0(e_{_{j_2}})$.\smallskip

Lemma \ref{l Te and Tv coincide on the intersection} proves that \begin{equation}\label{eq Tej coincide on the intersection} \hbox{$T_{e_{_{j_1}}} (x) = T_{e_{_{j_2}}} (x)$ for all $x\in E_0(e_{_{j_1}}) \cap E_0^{**}(e_{_{j_2}})$.}
\end{equation}

Let us define a mapping $T : E\to B$ given by $$T(x) = T_{e_{_{j_2}}} (x_{j_1}) + T_{e_{_{j_1}}} (x_{j_2}) + T_{e_{_{j_1}}} (x_0),$$ where $x = x_{j_1} + x_{j_2} + x_0$, $ x_{j_1}\in  K_{j_1},$ $x_{j_2}\in K_{j_2}$, $x_0\in \oplus^{c_0}_{j\neq j_1,j_2} K_j$, with respect to \eqref{eq decomposition of E orthogonals}. The mapping $T$ is well defined thanks to \eqref{eq Tej coincide on the intersection} and the uniqueness of the decomposition in \eqref{eq decomposition of E orthogonals}. The same argument and the real linearity of $T_{e_{j_1}}$ and $T_{e_{j_2}}$ prove that $T$ is real linear. Clearly $T$ is bounded with $\|T\|\leq 3$. \smallskip

Every finite rank tripotent in $E$ is of the form $u=u_{j_1}+ u_{j_2} + u_0$, where $u_{j_1}$, $u_{j_2}$ and $u_0$ are finite rank tripotents in $K_{j_1}$, $K_{j_2}$, and $\displaystyle \oplus^{c_0}_{j\notin \{j_1,j_2\}} K_j\subseteq  E_0(e_{_{j_1}}) \cap E_0(e_{_{j_2}})$, respectively. Now, by Corollary \ref{c rank preservation bewteen elementary} we have $$ f(u) = f(u_{j_1}+ u_{j_2} + u_0) = f(u_{j_1}+ u_{j_2})  + f(u_0) = f(u_{j_1}) + f(u_{j_2}) + f(u_0)$$ $$=\hbox{(by Corollary \ref{cor Te coincides with f on the orthogonal of e}) } = T_{e_{j_2}} (u_{j_1}) + T_{e_{j_1}} (u_{j_2}) + T_{e_{j_1}} (u_0) = T(u).$$
We have therefore shown that $f(u) = T(u)$ for every finite rank tripotent $u$ in $E$.
Finally Proposition \ref{p identity principle} assures that $T(x) = f(x)$ for every $x\in S(E)$, and in particular $T$ is a real linear surjective isometry.
\end{proof}

Many interesting consequences can be derived from the previous theorem. For example, suppose that $H_1$, $H_2$, $H_3$ and $H_4$, $H_1^{\prime}$, $H_2^{\prime}$, $H_3^{\prime}$, and $H_4^{\prime}$ are complex Hilbert spaces and $\displaystyle f : S(K(H_1,H_2)\oplus^{\infty} K(H_3,H_4))\to S(K(H_1^{\prime},H_2^{\prime})\oplus^{\infty} K(H_3^{\prime},H_4^{\prime}))$ is a surjective real linear isometry. Then there exists a real linear surjective isometry $T : K(H_1,H_2)\oplus^{\infty} K(H_3,H_4)\to K(H_1^{\prime},H_2^{\prime})\oplus^{\infty} K(H_3^{\prime},H_4^{\prime})$ such that $$T|_{S(K(H_1,H_2)\oplus^{\infty} K(H_3,H_4))} =f.$$

After Theorem \ref{thm Tyngley co sums with more than one element} above and the structure theory of weakly compact JB$^*$-triples (see Theorem \ref{thm compact triples BuncChu}), the study of surjective isometries between the unit spheres of two compact JB$^*$-triples can de reduced to the study of surjective isometries between the unit spheres of two elementary JB$^*$-triples.\label{commments after theorem}\smallskip

We are now in position to prove the second main result of this note.

\begin{theorem}\label{thm Tingley thm ofr weakly compact JB*-triples rank 5} Let $f: S(E) \to S(B)$ be a surjective isometry between the unit spheres of two weakly compact JB$^*$-triples not containing direct summands of rank smaller or equal than 3, or between two compact C$^*$-algebras. Suppose $E$ has rank greater or equal than 5. Then there exists a surjective real linear isometry $T: E\to B$ satisfying $T|_{S(E)} = f$.
\end{theorem}

\begin{proof} Since the rank of $E$ is greater or equal than $3$, we can find three mutually orthogonal minimal tripotents $e_1,e_2$ and $e_2$ in $E$. For each $j\in \{1,2,3\}$, let  $T_{e_j} : E_0^{**}(e_j) \cap E \to B_0^{**}(v_j) \cap B$ be the surjective real linear isometry given by Proposition \ref{p surjective isometries between the spheres preserve finite rank tripotents}$(c)$, where $v_j=f(u_j)$. Proposition \ref{p surjective isometries between the spheres preserve norm closed faces} and Corollary \ref{c rank preservation bewteen elementary} assure that $v_1,v_2$ and $v_3$ are mutually orthogonal minimal tripotents in $B$ and $f(e_1+e_2+e_3) = v_1 +v_2+v_3$.\smallskip

It is known that \begin{equation}\label{eq joint decompostion 1} E = \mathbb{C} e_1 \oplus (E_1 (e_1) \cap E_1(e_2))\oplus (E_1 (e_1) \cap E_0(e_2)) \oplus E_0 (e_1)
 \end{equation}(just compare the joint Peirce decomposition given in \cite[(1.12)]{Horn87}). Each $x$ in $E$ writes uniquely in the form $x= \lambda_x e_1 + x_{11}+ x_{10}+ x_0$, where $\lambda_x\in \mathbb{C}$, $x_{11}\in  (E_1 (e_1) \cap E_1(e_2))$, $x_{10}\in  (E_1 (e_1) \cap E_0(e_2)),$ and $x_0\in E_0 (e_1).$
We define a mapping $T: E\to B$ given by $$T(x) =T(\lambda_x e_1 + x_{11}+ x_{10}+ x_0) :=   T_{e_3} (\lambda_x e_1) + T_{e_3} (x_{11}) + T_{e_2}(x_{10})+ T_{e_1}(x_0),$$ where $x= \lambda_x e_1 + x_{11}+ x_{10}+ x_0$. The mapping $T$ is well defined and real linear thanks to the uniqueness of the decomposition in \eqref{eq joint decompostion 1} and the real linearity of $T_{e_1}$, $T_{e_2}$ and $T_{e_3}$.\smallskip

We claim that \begin{equation}\label{eq f and T coincide on rank one tripotents second thm} T(e) =f(e), \hbox{ for every minimal tripotent $e$ in $E$.}
\end{equation} Let $e$ be a minimal tripotent in $E$. Since $E$ has rank bigger or equal than 5, we can find $e_4$ satisfying $e_4 \perp e_1, e_2, e_3, e$.\smallskip

Since $e\in E_0(e_4)$, Corollary \ref{cor Te coincides with f on the orthogonal of e} implies that \begin{equation}\label{eq T4 and f coincide on e} f(e) = T_{e_4} (e).
\end{equation}

Le us write $e= \lambda_e e_1 + e_{11}+ e_{10}+ e_0$, where $\lambda_e\in \mathbb{C}$, $e_{11}\in  E_1 (e_1) \cap E_1(e_2)$, $e_{10}\in  (E_1 (e_1) \cap E_0(e_2)),$ and $e_0\in E_0 (e_1).$ Clearly, $\lambda_e e_1 \perp e_4$. By Peirce arithmetic $e_{11}\in  E_1 (e_1) \cap E_1(e_2) \subset E_2 (e_1+e_2) \perp e_4$, because $e_1+e_2\perp e_4$.\smallskip

Now, having in mind that $e\perp e_4$, we deduce that $$0=L(e_4,e_4) (e) = L(e_4,e_4) (\lambda_e e_1 + e_{11}+ e_{10}+ e_0) = L(e_4,e_4) (e_{10})+ L(e_4,e_4) (e_0).$$ A new application of Peirce arithmetic, gives $$ L(e_4,e_4) (e_{10}) \in E_1 (e_1), \hbox{ and } L(e_4,e_4) (e_0) \in E_0 (e_1),$$ and thus $$L(e_4,e_4) (e_0) = L(e_4,e_4) (x_{10})=0.$$ We have therefore shown that $\lambda_e e_1, e_{11}, e_{10}, e_0\in E_0(e_4)$. Applying Lemma \ref{l Te and Tv coincide on the intersection} to $T_{e_4}$ and $T_{e_3}$ (respectively, to $T_{e_4}$ and $T_{e_2}$, and $T_{e_4}$ and $T_{e_3}$) we obtain: $$T(e) = T_{e_3} (\lambda_x e_1) + T_{e_3} (x_{11}) + T_{e_2}(x_{10})+ T_{e_1}(x_0) $$ $$= T_{e_4} (\lambda_x e_1) + T_{e_4} (x_{11}) + T_{e_4}(x_{10})+ T_{e_4}(x_0) = T_{e_4} (e) = \hbox{ (by \eqref{eq T4 and f coincide on e}) } = f(e),$$ which concludes the proof of the claim.\smallskip

Let $u$ be a finite rank tripotent in $E$, and let us write $u= u_1+\ldots+u_m$, where $u_1,\ldots, u_m$ are mutually orthogonal minimal tripotents in $E$. By Corollary \ref{c rank preservation bewteen elementary} and \eqref{eq f and T coincide on rank one tripotents second thm} we have $$f(u)= f(u_1)+\ldots+f(u_m) = T(u_1)+\ldots+T(u_m)=T(u).$$ This shows that $T(u) = f(u)$ for every finite rank tripotent $u$ in $E$. Under these circumstances, Proposition \ref{p identity principle} implies that $T(x) = f(x)$, for every $s\in S(E)$, which concludes the proof.
\end{proof}

Let $A$ be a C$^*$-algebra. It is known that $A$ is a weakly compact JB$^*$-triple if and only if it is a compact JB$^*$-triple if and only if $A$ is a compact C$^*$-algebra in the sense employed in \cite{Alex,Yli}. It is therefore known that $A= \oplus^{c_0}_{j} K(H_j)$, where each $H_j$ is a complex Hilbert space. We can establish now a positive answer to Tingley's conjecture in the case of a surjective isometry between the unit spheres of two compact C$^*$-algebras.

\begin{theorem}\label{thm Tingley compact Cstaralgebras} Let $f: S(A)\to S(B)$ be a surjective isometry between the unit spheres of two compact C$^*$-algebras. Then there exists a surjective real linear isometry $T:A\to B$ such that $T(x) = f(x),$ for every $x\in S(A)$. In particular, the same conclusion holds when $A$ and $B$ are the C$^*$-algebras of compact operators on arbitrary complex Hilbert spaces.
\end{theorem}

\begin{proof} When $A$ and $B$ are finite dimensional, the conclusion follows from the main result in \cite[Theorem 4.8]{Tan2016preprint} or from \cite{Tan2016-2}.\smallskip

If $A$ (or $B$) is the $c_0$-sum of two or more $K(H_j)$, then the result follows from Theorem \ref{thm Tyngley co sums with more than one element} (compare the comments in page \pageref{commments after theorem}). We can thus assume that $f: S(K(H_1))\to S(K(H_2))$ is a surjective isometry. 
Since $H_1$ and $H_2$ are infinite dimensional the desired conclusion is a straight consequence of Theorem \ref{thm Tingley thm ofr weakly compact JB*-triples rank 5}.
\end{proof}

The results in this note cannot be applied in the study of surjective isometries between the unit spheres of two elementary JB$^*$-triples of rank $\leq  4$. This question is left as an open problem, although some additional information is known. Let $K$ be an elementary JB$^*$-triple of rank 1,2,3 or 4. The second half of the Table 1 in \cite[page 210]{Ka97} implies that $K$ is one of the following Cartan factors:
\begin{enumerate}[$(a)$]\item A Cartan factor of type 1 of the form  $L(H,K)$, with dim$(K)\leq 4$;
\item A Cartan factor of type 2, with dim$(H)\leq 9$;
\item A Cartan factor of type 3, with dim$(H)\leq 4$;
\item A spin factor whose rank is always 2.
\end{enumerate}

A Cartan factor $C$ of rank 1 must be type 1 Cartan factor of the form $L(H,\mathbb{C})$, where $H$ is a complex Hilbert space, or a type 2 Cartan
factor with dim $(H)=1$ or $3$ (it is known that $C$ is JB$^*$-triple isomorphic to a 1 or 3 dimensional complex Hilbert space). In any case, every rank-one Cartan factor is of the form $L(H,\mathbb{C})$. Let $H_1$ and $H_2$ be Hilbert spaces. G. Ding proved in \cite{Ding2002} that every isometric  surjective mapping $f: S(H_1)\to S(H_2)$ can be extended to a real linear isometric mapping $T: H_1\to H_2$. Combining these arguments we get:

\begin{corollary}\label{c Tingley between rank one Cartan factors} Let $f: S(H_1)\to S(H_2)$ be a surjective isometry between the unit spheres of two rank one JB$^*$-triples. Then there exists a surjective real linear isometry $T:H_1\to H_2$ such that $T(x) = f(x),$ for every $x\in S(H_1)$.
\end{corollary}

In the case of finite dimensional JB$^*$-triples, the arguments above can be adapted to obtain a generalization of \cite{Tan2016-2}.

\begin{proposition}\label{p surj isometries preserve rank} Let $f: S(E) \to S(B)$ be a surjective isometry between the unit spheres of two finite dimensional JB$^*$-triples. Suppose that $e_1$ and $e_2$ are two orthogonal finite rank tripotents in $E$. Then $f(e_1)$ and $f(e_2)$ are two orthogonal finite rank tripotents in $B$ with $f(e_1 + e_2 ) = f(e_1) + f(e_2)$. Furthermore, if $e$ is a rank $k$ tripotent in $E$ then $f(e)$ is a rank $k$ tripotent in $B$.
\end{proposition}

\begin{proof} Let us consider the tripotents $v = e_1-e_2$ and $\widetilde{v}= e_1+e_2$. Since $E$ and $B$ are finite dimensional, by Tingley's theorem, we know that $f(e_1-e_2) = f(u) = - f(-u) = - f(e_2-e_1)$ (see \cite{Ting1987}). By Proposition \ref{p surjective isometries between the spheres preserve finite rank tripotents}$(d)$, $w_1 = f(e_1)$, $w_2 = f(e_2)$, $w = f(e_1- e_2)= f(u)$ and $\widetilde{w} = f(e_1+e_2)$ are finite rank tripotents in $B$. Let $T_{e_1}$, $T_{e_2}$, $T_{v}$ and $T_{\widetilde{v}}$ denote the surjective real linear isometries given by Proposition \ref{p surjective isometries between the spheres preserve finite rank tripotents}$(c)$.

Since $e_1$, $v$ and $\widetilde{v}$ lie in the face $\{e_1\}_{\prime\prime}$ and $e_1 = 1/2 (v+ \widetilde{v})$, by Proposition \ref{p surjective isometries between the spheres preserve finite rank tripotents}$(b)$ we have $$ w_1 = f(e_1) = \frac12 (f(v) + f(\widetilde{v})) = \frac12 (w + \widetilde{w}) = \frac12 (f(v) + f(\widetilde{v}))= \frac12 (f(e_1-e_2) + f(e_1+e_2)) $$ $$= \frac12 (- f(e_2-e_1) + w_2 + T_{e_2} (e_1)) = \frac12 (- w_2 - T_{e_2} (-e_1) + w_2 + T_{e_2} (e_1)),$$ which shows that $T_{e_2} (e_1) = w_1 = f(e_1)$. Therefore $$\widetilde{w}= f(e_1+e_2) = w_2 + T_{e_2} (e_1) = w_2 + w_1 = f(e_2) + f(e_1).$$

Replacing $e_2$ with $-e_2$ we prove $${w}= f(e_1-e_2) = f(e_1) + f(-e_2) = f(e_1) -f(e_2) = w_1 - w_2.$$ we have therefore shown that $w_1\pm w_2$, $w_1$ and $w_2$ are tripotents in $B$, and hence $w_1\perp w_2$ (compare \eqref{eq orthogonality via pm}).\smallskip

Since by Proposition \ref{p surjective isometries between the spheres preserve norm closed faces}$(c)$ $f$ maps minimal tripotents to minimal tripotents, we deduce from the first statement that $f$ maps tripotents of rank $k$ to tripotents of rank $k$.
\end{proof}

When in the proof of Corollary \ref{cor Te coincides with f on the orthogonal of e} and Lemma \ref{l Te and Tv coincide on the intersection}, we replace Corollary \ref{c rank preservation bewteen elementary} with Proposition \ref{p surj isometries preserve rank} we obtain:

\begin{corollary}\label{cor Te coincides with f on the orthogonal of e finite dimensional} Let $f: S(E) \to S(B)$ be a surjective isometry between the unit spheres of two finite dimensional JB$^*$-triples. Let $e$ be a finite rank tripotent in $E$, $u= f(e)$, and let $T_e : E_0(e) \to B_0(u)$ be the surjective real linear isometry given by Proposition \ref{p surjective isometries between the spheres preserve finite rank tripotents}$(c)$. Then $f(x) = T_e (x)$ for all $x\in S(E_0(e))$.$\hfill\Box$
\end{corollary}

\begin{lemma}\label{l Te and Tv coincide on the intersection finite dimensional} $f: S(E) \to S(B)$ be a surjective isometry between the unit spheres of two finite dimensional JB$^*$-triples. Let $e_1$ and $e_2$ be two orthogonal finite rank tripotents in $E$, and let $T_{e_1}$ and $T_{e_2}$ be the maps given by Proposition \ref{p surjective isometries between the spheres preserve finite rank tripotents}$(c)$. Then $$\hbox{$T_{e_{1}} (x) = T_{e_{2}} (x)$ for all $x\in E_0(e_{1}) \cap E_0(e_{2})$.}$$ $\hfill\Box$
\end{lemma}

An obvious adaptation of the arguments in the proof of Theorem \ref{thm Tingley thm ofr weakly compact JB*-triples rank 5 finite dimensional} can be now applied to establish a generalization of the main result in \cite{Tan2016-2}.

\begin{theorem}\label{thm Tingley thm ofr weakly compact JB*-triples rank 5 finite dimensional} Let $f: S(E) \to S(B)$ be a surjective isometry between the unit spheres of two finite dimensional JB$^*$-triples. Suppose $E$ has rank greater or equal than 5. Then there exists a surjective real linear isometry $T: E\to B$ satisfying $T|_{S(E)} = f$. $\hfill\Box$
\end{theorem}

We do not know the answer in the remaining cases.

\begin{problem} Let $C$ and $B$ be Cartan factors of rank $\leq 4$. Suppose $f: S(C) \to S(B)$ is a surjective isometry. Does $f$ extend to a surjective real linear isometry from $C$ into $B$?
\end{problem}\smallskip

\textbf{Acknowledgements} First author partially supported by the Spanish Ministry of Economy and Competitiveness and European Regional Development Fund project no. MTM2014-58984-P and Junta de Andaluc\'{\i}a grant FQM375. The second author was supported in part by Grants-in-Aid for Scientific Research Grant Numbers 16J01162, Japan Society for the Promotion of Science.

\end{document}